\documentclass[10pt]{amsart}

\usepackage{amsmath,amsfonts,amsthm,amstext,amscd}
\usepackage{hyperref}
\usepackage[cmtip,all]{xy}

\def\cal#1{\mathcal{#1}}

\def\lr#1{\left\langle #1 \right\rangle}

\def\co{\colon}

\newtheorem{theorem}{Theorem}[section]

\newtheorem{conj}{Conjecture}

 \newtheorem{lem}[theorem]{Lemma}

\theoremstyle{definition}
\newtheorem{definition}[theorem]{Definition}

\theoremstyle{remark}

\numberwithin{equation}{section}
 
 \numberwithin{equation}{section}

\begin{document}

\title[An intrinsic curvature condition for submersions]{An intrinsic curvature condition for submersions over Riemannian manifolds}


\author{Llohann D. Speran\c ca}
\address{Instituto de Ci\^encia e Tecnologia, Unifesp\\
	Av. Cesare Monsueto Giulio Lattes, 1211 - Jardim Santa Ines I\\
	CEP 12231-280 \\
	S\~ao Jos\'e dos Campos, SP, Brazil}
\email{lsperanca@gmail.br}




\begin{abstract}Let $\pi{:}\,(M,\cal {H})\to (B,b)$ be a submersion equipped with a horizontal connection $\cal H$ over a Riemannian manifold $(B,b)$. We present an intrinsic curvature condition that only depends on the pair $(\cal H,b)$. By studying a set of relative flat planes, we prove that a certain class of pairs $(\cal H,b)$ admits a compatible metric with positive sectional curvature only if they are  \textit{fat}, verifying Wilhelm's Conjecture in this class. \end{abstract}

\maketitle

\section{Introduction}

Submersions from Lie groups constitute a main source of examples for manifolds with positive sectional curvature (see \cite{ziller2007examples} for a survey). Given a Riemannian submersion $\pi\co (M,g)\to (B,b)$ (i.e., a submersion $\pi\co M\to B$ such that the restriction $d\pi_{(\ker d\pi)^\bot}$ is an isometry), an usual practice is to deform the  starting metric $g$ along directions tangent to $\ker d\pi$. Here we study a curvature condition  independent of such deformations (equation \eqref{eq:WNN}, Proposition \ref{thm:var}).

Let $\pi\co (M,\cal H)\to (B,b)$ be a submersion equipped with a horizontal connection $\cal H$, i.e., a distribution $\cal H$ complementary to $\cal V=\ker d\pi$. We call a Riemannian metric $g$ on $M$ \textit{adapted to $(\cal H,b)$} if $g(\cal V,\cal H)=0$ and $d\pi|_\cal H\co\cal H\to TB$ is an isometry (given $(\cal H,b)$, the set of adapted metrics is in bijection with the set of metrics in $\cal V$, which is usually infinite dimensional.) One verifies that $\pi\co (M,g)\to (B,b)$ is Riemannian for any adapted metric $g$. Using the theory developed in this note  we prove, for instance, the following result about homogeneous submersions (a broader class of spaces and metrics is considered in Theorem \ref{thm:var}):
\begin{theorem}\label{thm:homo}
Let $K{<\,}H{<\,}G$ be compact Lie groups and consider the submersion $\pi\co (G/K,\cal H)\to (G/H,b)$ defined by $\pi(lK)=lH$. Suppose that $(G/H,b)$ is  normal homogeneous and $l\cdot \cal H_p=\cal H_{lp}$ for all $l\in G$. Then $G/K$ has an adapted metric with positive sectional curvature only if $\cal H$ is fat.
\end{theorem}

We call a connection $\cal H$ as \textit{fat} if every non-zero vector $X\in\cal H$ can be locally extended to a field $\bar X\in\cal H$ such that $\cal V=[\bar X,\cal H]^v$ (see e.g. \cite{weinstein1980fat}). 

Gonz\'alez and Radeschi \cite{gonzalez2017note}  considered submersions from spaces homotopically equivalent to the known examples with positive sectional curvature. They prove that such submersions satisfy Wilhelm's Conjecture conclusion, without curvature assumptions:

\begin{conj}
	Let $\pi\co (M^{n+k},g)\to (B^n,b)$ be a Riemannian submersion. If $(M,g)$ is compact and has positive sectional curvature, then $k<n$.
\end{conj}

(See also  \cite{amann2015positive,chen2016riemannian,gonzalez2016soft} for interesting developments.)
Theorem \ref{thm:homo} assumes hypothesis on curvature but might cover different spaces. For instance, let $G=Spin(5)$, $H=Spin(3)$ and $G/H$ be the Berger space $B^7$ (\cite{berger1961varietes}). Then Theorem \ref{thm:homo} guarantees that $Spin(5)$ do not have a positively curved metric adapted to $(\cal H,b)$,  the usual horizontal distribution and metric induced by the bi-invariant metric. Theorem \ref{thm:var} below covers a wider class of submersions, including the non-standard submersions constructed in \cite{kerin-shankar}.

Generally, the existence of a fat connection implies Wilhelm's conjecture. In particular, Theorems \ref{thm:homo} and Theorem \ref{thm:var} verifies Wilhelm's conjecture for  particular classes of metrics.

Our strategy is to explore the behavior of the set of vertizontal planes that vanishes under Grey--O'Neill's $A$-tensor, assuming  a certain curvature condition: given $X,Y$, horizontal vectors at $p\in M$, one computes $A\co\cal H\times\cal H\to \cal V$ as
\begin{equation}\label{eq:A}
A_XY=\frac{1}{2}[\bar X,\bar Y]^v,
\end{equation}
where $\bar X,\bar Y\in\cal H$ are horizontal extensions of $\bar X,\bar Y$ and $^v$  stands for the projection onto the $\cal V$-component. We also consider $A^*\co\cal H\times\cal V\to \cal H$,  the dual of $A$, defined as $\lr{A^*_X\xi,Y}=\lr{A_XY,\xi}$ for all $X,Y\in\cal H$ and $\xi\in\cal V$.

\begin{definition}
Let $\pi\co(M,\cal H)\to (B,b)$ be a submersion. An adapted metric $g$ is called weakly non-negatively curved (WNN) if every $p\in M$ has a neighborhood $U$ where 
\begin{equation}\label{eq:WNN}
\tau ||X||||A^*_X\xi||^2\geq\lr{(\nabla_XA^*)_X\xi+A_X^*S_X\xi ,A^*_X\xi},
\end{equation}
for some $\tau>0$ and all $X\in\cal H|_U$, $\xi\in\cal V|_U$.
\end{definition}

Inequality \eqref{eq:WNN} holds on submersion with totally geodesic fibers from a non-negatively curved manifold, being a condition strictly weaker then non-negative sectional curvature -- observe that \eqref{eq:WNN} does not give any estimate on the curvature of the base, in contrast to non-negative curvature (see \cite{oneill}). Furthermore, the WNN property is intrinsic to $(\cal H,b)$ (item (1) below).

\begin{theorem}\label{thm:var}
	Let $\pi\co(M,\cal H)\to (B,b)$ be a submersion and $g$ be adapted to $(\cal H,b)$. Then,
\begin{enumerate}
	\item 	$g$ is weakly non-negatively curved if and only if every metric adapted to $(\cal H,b)$ is weakly non-negatively curved;
	\item  if $g$ has  non-negative sectional curvature and the fibers of $\pi$ are totally geodesic, then $g$ is weakly non-negatively curved;
	\item suppose $(M,g)$ is compact, weakly non-negatively curved and $\pi$ has a compact structure group. Then, if $g$ has positive sectional curvature, $\cal H$ is fat.
\end{enumerate}
\end{theorem}

Based on item (1), we call a pair $(\cal H,b)$ \textit{weakly non-negatively curved} if there is a WNN metric $g$ adapted to $(\cal H,b)$. Theorem \ref{thm:homo} follows from Theorem \ref{thm:var} by observing that $G/K\to G/H$ has a metric as in $(2)$.  Item $(3)$ is proved through a flat-strip type of result (Lemma \ref{lem:flatgeo}).

\vspace*{0.2cm}

The author would like to thank David Gonz\'alez for his comments on the paper.

%

\section{Dual holonomy fields}
The main tool used here  is {dual holonomy fields} (as introduced in \cite{speranca_oddbundles}). Let $c$ be a horizontal curve, i.e., $\dot c\in\cal H$, recall that a \textit{holonomy field} $\xi$ along $c$ is a vertical field satisfying
\begin{equation*}
\nabla_{\dot c}\xi=-A^*_{\dot c}\xi-S_{\dot c}\xi,
\end{equation*}
where $S\co \cal H\times\cal H\to \cal V$ is the \textit{second fundamental form of the fibers:}
\begin{equation*}
S_X\xi=-(\nabla_\xi \bar X)^v
\end{equation*}
with $\bar X$ any horizontal extension of $X$. A vertical field $\nu$ along $c$ is called a \textit{dual holonomy field} if\begin{equation}\label{eq:dualhol}
\nabla_{\dot c}\nu=-A^*_{\dot c}\nu+S_{\dot c}\nu.
\end{equation}

As we shall see, dual holonomy fields naturally appear when dealing with adapted metrics. Their role is represented by Lemmas \ref{lem:dualinv}, \ref{lem:flatgeo} and equation \ref{eq:K}. 

Given a horizontal geodesic $c$, we define its \textit{infinitesimal holonomy transofrmation} $\hat{c}(t)\co\cal V_{c(0)}\to\cal V_{c(t)}$ by setting $\hat c(t)\xi(0)=\xi(t)$, where $\xi$ is a holonomy field along $c$.
We recall  that, given an adapted metric $g$, the dual holonomy field  $\nu$ along $c$ satisfies (\cite[Proposition 4.1]{speranca_oddbundles}):
\begin{equation*}
\nu(t)=\hat c(t)^{-*}\nu(0),
\end{equation*}
where $\hat c(t)^{-*}$ is the inverse of the $g$-dual of $\hat c(t)$. Let $g,g'$ be metrics adapted to $(\cal H,b)$ and let $P_p\co\cal V_p\to\cal V_p$ be the tensor defined by
\begin{equation}
g'(\xi,\eta)=g(P_p\xi,\eta).
\end{equation}
We use $^*$ to denote $g$-duals and $^\dagger$ to denote $g'$-duals. Thus, $T^\dagger=P^{-1}_pT^*P_q$ for any operator $T\co\cal V_p\to \cal V_q$.

\begin{lem}\label{lem:dualinv}
	If $\nu'(t)$ is a dual holonomy field on $(M,g')$,  then $\nu'(t)=P^{-1}_{c(t)}\nu(t)$, where $\nu(t)$ is the dual holonomy  field on $(M,g)$  with initial condition $\nu(0)=P_{c(0)}\nu'(0)$. In particular, for every $t$,
	\begin{equation}\label{eq:lemdualinv}
	A^\dagger_{\dot c}\nu'(t)=A^*_{\dot c}\nu(t).
	\end{equation}
\end{lem}
\begin{proof}
According to the discussion above, 
\[\nu'(t)=\hat c(t)^{-\dagger}\nu'(0)=(\hat c(t)^\dagger)^{-1}\nu'(0)=P^{-1}_{c(t)}\hat c(t)^{-*}P_{c(0)}\nu'(0)=P^{-1}_{c(t)}\nu(t).\]
Equation \eqref{eq:lemdualinv} follows since $A^\dagger_X\xi=A^*_XP_p\xi$ for all $X\in\cal H_p,\xi\in\cal V_p$.
\end{proof}
\begin{lem}\label{lem:flatgeo}
	Let $g$ be a WNN metric adapted  to $(\cal H,b)$. If $\nu$ is a dual holonomy field along a horizontal geodesic $c$ satisfying $A^*_{\dot c}\nu(0)=0$, then $A^*_{\dot c}\nu(t)=0$ for all $t$.
\end{lem}
\begin{proof}
	The proof goes along the lines of Proposition 5.2 in \cite{speranca_grove}. Take $\|\dot c\|=1$ and denote $u(t)=||A^*_{\dot c}\nu(t)||^2$. From \eqref{eq:dualhol}, we have
	\begin{align*}
	(\nabla_{\dot c}A^*)_{\dot c}\nu=\nabla_{\dot c}(A^*_{\dot c}\nu)-A^*_{\dot c}\nabla_{\dot c}^v\nu=\nabla_{\dot c}(A^*_{\dot c}\nu)+A_{\dot c}^*S_{\dot c}\nu.
	\end{align*}
	Thus, equation \eqref{eq:WNN} gives $2\tau u(t)\geq u'(t)$. Gronwall's inequality now implies that $u(t)\leq u(0)e^{2\tau t}$ for all $t\geq 0$. Replacing $c$ by $\tilde c(t)=c(-t)$ proves the assertion. 
\end{proof}
As a last observation, we recall an identity in \cite{speranca_oddbundles}. Given a Riemannian metric $g$, let $K_g(X,\nu)=R_g(X,\nu,\nu,X)$ be the unreduced sectional curvature of $X\wedge\nu$. We recall from  \cite[Proposition 4.2]{speranca_oddbundles} that a dual holonomy field $\nu$ satisfies
\begin{gather}\label{eq:K}
	K_g(\dot c,\nu(t))=\frac{1}{2}\frac{d^2}{dt^2}||\nu(t)||^2-3||S_{\dot c}\nu(t)||^2+||A^ *_{\dot c}\nu(t)||^2.
\end{gather}

\section{Proof of Theorem \ref{thm:var}}
\begin{proof}[Proof of $(1)$]
	Let $g,g'$ be adapted to $(\cal H,g)$ and assume $g$ WNN. Therefore, given $p\in M$, there is a neighborhood $U$ of $p$ and a $\tau>0$ such that \eqref{eq:WNN} holds. For any metric $h$, $X\in\cal H_q$ and $\nu_0\in\cal V_q$, denote $u(h,\nu_0,X,t)=\|A^\vee_{\dot c}\nu(t)\|^2$ where $c(t)=\exp(tX)$, $\nu$ is the dual holonomy field along $c$ with $\nu(0)=\nu_0$ and $A^\vee$ is the $h$-dual of $A$. Lemma \ref{lem:dualinv} gives $u(g',\nu_0,X,t)=u(g,P_q\nu_0,X,t)$ for every  $X\in\cal H_q$, $\nu_0\in\cal V_q$ and $t$. In particular,
\[2\tau u(g',\nu_0,X,t)=2\tau u(g,P_q\nu_0,X,t)\geq u'(g,P_q\nu_0,X,t)=u'(g',\nu_0,X,t).\qedhere\]
\end{proof}

\begin{proof}[Proof of  $(2)$]
Item $(2)$ follows by analyzing the discriminant of the polynomial 
\[K_g( X,\lambda A^*_X\xi+\xi)=K_g(\xi,X)+2\lambda R_g(X,A^*_X\xi,\xi,X)+\lambda^2K_g(X,A^*_X\xi).\]
From the curvature assumption, we have 
\[K_g(\xi,A^*_X\xi)K_g(X,A^*_X\xi)\geq R_g(X,A^*_X\xi,\xi,X)^2.\] 
On the other hand,  if $(M,g)$ has totally geodesic fibers, O'Neill's equation (see \cite{oneill} or \cite[page 44]{gw}) gives $R_g(X,A^*_X\xi,\xi,X)=\lr{(\nabla_XA^*)_X\xi,A^*_X\xi}$ and $K_g(\xi,X)=\|A^*_X\xi\|^2$. Moreover, taking a relatively compact neighborhood on $M$, there is a $\tau>0$ such that $K_g(X,A^*_X\xi)\leq \tau \|X\|^2\|A^*_X\xi\|^2$.
\end{proof}

\begin{proof}[Proof of  $(3)$]
	We argue by contradiction. Suppose there are   non-zero $X\in\cal H,\nu_0\in\cal V$ such that $A^*_X\nu_0=0$. On one hand, Lemma \ref{lem:flatgeo} guarantees that $A^*_{\dot c}\nu(t)=0$ for the dual holonomy field $\nu$ along $c$ with $\nu(0)=\nu_0$. On the other hand, holonomy fields have bounded norm whenever $\pi$ admits a compact structure group, i.e., there is a constant $L$ such that $\|\nu(t)\|^2\leq L\|\nu(0)\|^2$ for every $t$ (see \cite[Proposition 3.4]{speranca_oddbundles}). Applying equation \eqref{eq:K}, we get 
\begin{gather*}\label{eq:KK}
\frac{d^2}{dt^2}\|\nu(t)\|^2\geq 2K_g(\dot c,\nu)\geq \kappa \|\nu(t)\|^2
\end{gather*} 
for some fixed constant $\kappa>0$, contradicting  the boundedness of $\|\nu(t)\|^2$.
\end{proof}

\bibliographystyle{alpha}
\bibliography{bib230617}

\end{document}